\documentclass[a4paper,11pt]{article}

\usepackage[top=3.0cm,bottom=3.0cm,left=2.5cm,right=2.5cm]{geometry}
\usepackage{tikz}
\usepackage{graphics,epsfig,psfrag}
\usepackage{graphicx}
\usepackage{amsfonts}
\usepackage{amssymb}
\usepackage{amsmath}
\usepackage{amsthm}
\usepackage[english]{babel}
\usepackage{ctex}
\usepackage{hyperref}
\usepackage{cases}
\usepackage{authblk}
\usepackage{mathtools}
\usepackage{pifont}

\newtheorem{lemma}{Lemma}[section]
\newtheorem{theorem}[lemma]{Theorem}

\newtheorem{claim}[]{\noindent Claim}[section]

\newtheorem{conjecture}[lemma]{Conjecture}

\begin{document}

\title{Seymour and Woodall's conjecture holds for graphs with independence number two}

        \author[1]{Rong Chen\footnote{Email: rongchen@fzu.edu.cn (Rong Chen).}}
		\author[1]{Zijian Deng\footnote{Email: zj1329205716@163.com.(corresponding author).}}

\affil[1]{Center for Discrete Mathematics, Fuzhou University\\

Fuzhou, People's Republic of China}

	\date{}
	\maketitle
	\begin{abstract}
Woodall (and Seymour independently) in 2001 proposed a conjecture that every graph $G$ contains every complete bipartite graph on $\chi(G)$ vertices as a minor, where $\chi(G)$ is the chromatic number of $G$. In this paper, we prove that for each positive integer $\ell$ with $2\ell \leq \chi(G)$, each graph $G$ with independence number two contains a $K^{\ell}_{\ell,\chi(G)- \ell}$-minor, implying that Seymour and Woodall's conjecture holds for graphs with independence number two, where $K^{\ell}_{\ell,\chi(G)- \ell}$ is the graph obtained from  $K_{\ell,\chi(G)- \ell}$ by making every pair of vertices on the side of the bipartition of size $\ell$ adjacent.


	\end{abstract}
	
	


\section{\bf Introduction}
All graphs considered in this paper are finite and simple. 
For a graph $G$, we use $\chi(G)$, $\omega(G)$ and $\Delta(G)$ to denote the chromatic number, clique number and maximum degree of $G$, respectively.
An \emph{independent set} 
is a subset of $V(G)$ that are pairwise nonadjacent. Let $\alpha(G)$ be the \emph{independence number} of $G$, that is the maximum size of independent sets. Let $K_{n}$ be a clique on $n$ vertices, $P_{n}$ a path on $n$ vertices. 
If a graph isomorphic to $H$ can be obtained from a graph $G$ by deleting vertices or edges and contracting edges, then $H$ is a \emph{minor} of $G$, written by $G\succeq_m H$. 

In 1943, Hadwiger \cite{HU} proposed the following famous conjecture.

\begin{conjecture}{\rm(Hadwiger's conjecture)}\label{conj1}
 Every graph $G$ contains $K_{\chi(G)}$ as a minor.
\end{conjecture}

Hadwiger's conjecture has been proven to be true only for graphs with chromatic number at most 6, but is likely difficult to prove for larger values, given that proofs for the cases of $\chi(G)=5$ and $\chi(G)=6$ already depend on the Four Color Theorem. Hadwiger's conjecture remains unsolved for a specific class of graphs, namely those with an independence number of at most two.


\begin{conjecture}\label{conj2}
Every graph $G$ with $\alpha(G)\leq2$ contains $K_{\chi(G)}$ as a minor.
\end{conjecture}

The significance of this class of graphs with independence number two was first pointed out by Duchet and Meyniel \cite{PH} and Mader independently, and later highlighted in Seymour's survey paper \cite{P}.
According to Seymour in \cite{P}, 
Conjecture \ref{conj2} could be key to the Hadwiger's conjecture. He suggested that if the conjecture is valid for these graphs, it might be valid more broadly.
Plummer, Stiebitz, and Toft \cite{MM} showed that, Conjecture \ref{conj2} can be restated as follows.

\begin{conjecture}\label{conj3}
Every $n$-vertex graph $G$ with $\alpha(G)\leq2$ contains $K_{\lceil \frac{n}{2}\rceil}$ as a minor.
\end{conjecture}

Chudnovsky in \cite{M} showed that every $n$-vertex graph $G$ with an independence number $\alpha(G)\leq2$ contains a $K_{\lceil \frac{n}{3}\rceil}$-minor, which is also a direct consequence of Duchet and Meyniel's result in \cite{PH}. This result provides strong evidence for Conjecture \ref{conj3}.
Recently, Botler et al. \cite{FA} proposed a conjecture that is slightly weaker than Conjecture \ref{conj3}.

\begin{conjecture}\label{conj4}{\rm(\cite{FA})}
Let $G$ be an $n$-vertex graph with $\alpha(G)\leq2$. For any positive integer $\ell$ with $\ell < \lceil\frac{n}{2}\rceil$, we have $G \succeq_{m} K_{\ell,\lceil\frac{n}{2}\rceil- \ell}$.
\end{conjecture}

In fact, Woodall \cite{W} in 2001 (and independently Seymour \cite{AN}) proposed a conjecture that is slightly weaker than Hadwiger's conjecture and stronger than Conjecture \ref{conj4}.

\begin{conjecture}{\rm(\cite{W,AN})}\label{conj1w}
For any positive integer $\ell$ with $\ell< \chi(G)$, a graph $G$ contains a $K_{\ell, \chi(G)-\ell}$-minor.
\end{conjecture}

In this paper, we show that Conjecture \ref{conj4} holds  and Conjecture \ref{conj1w} and Conjecture \ref{conj4} are equivalent for graphs with independence number two. 
In fact, we prove a stronger result than  Conjecture \ref{conj4}. To state our result, we need one more definition. Let $K_{m,n}^{m}$ be the graph obtained from a complete bipartite graph $K_{m,n}$ by making every pair of vertices on the side of the bipartition of size $m$ adjacent. In other words, $K_{m,n}^m$ is obtained from the disjoint union of a $K_m$ and an independent set on $n$ vertices by adding all of the possible edges between them.
In this paper, we prove the following result, implying that Conjecture \ref{conj4} holds and Conjecture \ref{conj1w} holds for graphs with independence number two.
\begin{theorem} \label{thm4} Let $G$ be a graph with $\alpha(G)\leq2$. For any positive integer $\ell$ with $2\ell \leq \chi(G)$, we have $G \succeq_{m} K^{\ell}_{\ell,\chi(G)- \ell}$.
\end{theorem}



\section{Proof of Theorem \ref{thm4}} 
To prove Theorem \ref{thm4}, we need to introduce more definitions and results. For a graph $G$, let  $\kappa(G)$ be the connectivity of $G$. For any set $X\subseteq V(G)$, let $G[X]$ denote the subgraph of $G$ induced by $X$. For any set $\mathcal{A}$ of vertex disjoint subgraphs of $G$, let $V(\mathcal{A})$ denote the union of the vertex sets of subgraphs in $\mathcal{A}$.
Given disjoint vertex sets $A$ and $B$, we say that $A$ is \emph{complete} to $B$ if  each vertex in $A$ is adjacent to all vertices in $B$. Analogously, $A$ is \emph{anticomplete} to $B$ if there exists no edge between $A$ and $B$. We say $A$ is \emph{mixed}  on $B$ if $A$ is neither complete nor anticomplete to $B$.
We denote by $N_{G}(A)$ the vertex set in $V(G)-A$ that has adjacent vertices in $A$.
Set $N_{G}[A]:=N_{G}(A)\cup A$.
When there is no danger of confusion, all subscripts will be omitted.
For simplicity, when $A = \{a\}$, the sets $N(\{a\})$ and $N[\{a\}]$ are denoted by $N(a)$  and $N[a]$, respectively.
We say a matching in the complement graph of $G$ is an \emph{anti-matching} in $G$.

For a nonempty clique $C$ in $G$, let $A,B$, and $D$ be the sets of vertices in $V(G)-C$ that are complete to $C$, anticomplete to $C$, and mixed on $C$ respectively. Evidently, $(A,B,C,D)$ is a partition of $V(G)$. We define the \emph{capacity} of $C$, written by $cap(C)$, to be $|D|+\frac{|A\cup B|}{2}$.
A \emph{five-wheel} is a six-vertex graph obtained from a cycle of length five by adding one new vertex adjacent to every vertex of the cycle.

\begin{theorem} \label{thm0}{\rm(\cite{J})} Let $G$ be an $n$-vertex graph with $\alpha(G)\leq2$. If $G$ is not $\lceil\frac{n}{2}\rceil$-connected, then $G\succeq_{m} K_{\lceil\frac{n}{2}\rceil}$.
\end{theorem}

\begin{theorem} \label{thm2}{\rm(\cite{MP})} Let $G$ be an $n$-vertex graph with $\alpha(G)\leq2$. If $n$ is odd and $\omega(G)\geq \frac{n+3}{4}$, then $G\succeq_{m} K_{\lceil\frac{n}{2}\rceil}$.
\end{theorem}

\begin{theorem} \label{thm1}{\rm(\cite{MP})} Let $G$ be a graph with $\alpha(G)\leq2$, and let $\ell\geq0$ be an integer, such that if $\ell=2$ then $G$ is not a five-wheel. Then $G$ has $\ell$ pairwise disjoint induced subgraphs, each isomorphic to a $3$-vertex-path $P_3$ if and only if

{\rm(1)} $|V(G)|\geq 3\ell$,

{\rm(2)} $G$ is $\ell$-connected,

{\rm(3)} every clique of $G$ has capacity at least $\ell$, and

{\rm(4)} $G$ admits an anti-matching of cardinality $\ell$.

\end{theorem}


\begin{lemma} \label{lem1}
Let $G$ be an $n$-vertex  graph with $\alpha(G)\leq2$, $\omega(G)<\lceil\frac{n}{2}\rceil$, and with $\kappa(G)\geq \lceil\frac{n-1}{4}\rceil$. For any integer $\ell$ with $\ell \leq \lceil\frac{n}{4}\rceil$, if $n\geq 4\ell+1$ is odd, then $G$ has $\ell$ pairwise disjoint induced $3$-vertex-paths $P_3$.
\end{lemma}
\begin{proof}
Since $n$ is odd, $G$ is not a five-wheel. Hence, to prove the lemma, by Theorem \ref{thm1}, it suffices to prove that $G$ satisfies Theorem \ref{thm1} (1)-(4). Theorem \ref{thm1} (1) and (2) are obviously true for $G$ as $\ell \leq \frac{n-1}{4}$.
Let $C$ be a clique of $G$. Since $\omega(G)<\lceil\frac{n}{2}\rceil$, we have $|C|< \lceil\frac{n}{2}\rceil$, so
$|V(G\backslash C)|\geq \lceil\frac{n}{2}\rceil \geq 2\ell+1$. Thus $cap(C)\geq \lceil\frac{2\ell+1}{2}\rceil >\ell$, implying that Theorem \ref{thm1} (3) holds.

We claim that Theorem \ref{thm1} (4) holds for $G$. Suppose not. Let $A$ be a maximal anti-matching of $G$ with $|V(A)|$ as large as possible. Then $G\backslash V(A)$ is a clique and $|V(A)|\leq2(\ell-1)$. Hence,
$$\omega(G)\geq|V(G)-V(A)|\geq n-2(\ell-1)\geq \lceil\frac{n}{2}\rceil$$ as $2\ell\leq \lceil\frac{n}{2}\rceil$, which is a contradiction. So Theorem \ref{thm1} (4) holds for $G$.
\end{proof}


\begin{theorem} \label{thm3.5}
Let $G$ be an $n$-vertex graph with $\alpha(G)=2$. For any positive integer $\ell$ with $2\ell \leq \lceil\frac{n}{2}\rceil$, we have $G \succeq_{m} K^{\ell}_{\ell,\lceil\frac{n}{2}\rceil- \ell}$.
\end{theorem}
\begin{proof}
Assume that Theorem \ref{thm3.5} is not true. Let $G$ be a counterexample to Theorem \ref{thm3.5} with $|V(G)|$ as small as possible. 


\begin{claim}\label{claim0}
$n\geq4\ell-1$ is odd.
\end{claim}

\begin{proof}
Assume that $n$ is even. Since Theorem \ref{thm3.5} holds for $G \backslash v$ for any $v\in V(G)$, Theorem \ref{thm3.5} holds for $G$ as $\lceil\frac{n}{2}\rceil=\lceil\frac{n-1}{2}\rceil$, a contradiction. Therefore, $n$ is odd. Since $\ell \leq \lceil\frac{n}{2}\rceil- \ell=\frac{n+1}{2}-\ell$, we have $n\geq4\ell-1$.
\end{proof}


Since $G$ is a counterexample to Theorem \ref{thm3.5}, we have $\kappa(G)\geq \lceil\frac{n}{2}\rceil$ by Theorem \ref{thm0}. Moreover, since $n$ is odd by Claim \ref{claim0}, we have $\omega(G)< \frac{n+3}{4}$ by Theorem \ref{thm2}.

\begin{claim}\label{claim1}
$n=4\ell-1$.
\end{claim}
\begin{proof}
Suppose not. Since $n$ is odd by Claim \ref{claim0}, we have $n\geq4\ell+1$. Since $\kappa(G)\geq \lceil\frac{n}{2}\rceil \geq \lceil\frac{n-1}{4}\rceil $
and $\omega(G)< \frac{n+3}{4} \leq \lceil\frac{n}{2}\rceil$,
it follows from Lemma \ref{lem1} that 
$G$ contains $\ell$ pairwise disjoint induced subgraphs, each isomorphic to $P_3$, denoted the collection of such subgraphs  by $\mathcal{P}$.
Let $H$ be the graph obtained from $G$ by contracting each induced path of $\mathcal{P}$ to a vertex. Since $\alpha(G)=2$, each vertex in $H$ obtained by contracting each member of $\mathcal{P}$ is complete to all other vertices in $H$. Moreover, since $n\geq 4\ell+1$, we have $|V(G)-V(\mathcal{P})|=n-3\ell\geq \lceil\frac{n}{2}\rceil -\ell$, so $H$ contains a subgraph isomorphic to $K^{\ell}_{\ell,\lceil\frac{n}{2}\rceil- \ell}$, which is a contradiction.
\end{proof}


Since $\omega(G)< \frac{n+3}{4}$ and $n=4\ell-1$ by Claim \ref{claim1}, we have $\omega(G)\leq \ell$.

For any edge $uv\in E(G)$, set $M_{uv}:=V(G)-N[\{u,v\}]$.
Since $\alpha(G)=2$, the set $M_{uv}$ induces a clique in $G$. Hence, $|M_{uv}|\leq \ell$ as $\omega(G)\leq \ell$.

\begin{claim}\label{claim3}
There exists an edge $uv$ such that $|M_{uv}|\leq \ell-1$.
\end{claim}
\begin{proof}
Let $x,y\in V(G)$ with $xy\notin E(G)$. Since $\alpha(G)=2$, such $x,y$ exist.
Set \[C := N(x)\cap N(y),\ X:=V(G)-N[y],\ Y:= V(G)-N[x].\] Evidently,  $x\in X$ and $y\in Y$. Since $\alpha(G)=2$, we have that $(C,X,Y)$ is a partition of $V(G)$, and $G[X]$ and $G[Y]$ are cliques.
Moreover, since $\omega(G)\leq \ell$, we have $|Y|\leq \ell$. For any vertex $c \in C$, since $C\cup X\cup\{y\}\subseteq N[\{c,x\}]$, we have $M_{cx}\subseteq Y-\{y\}$, so $|M_{cx}|\leq \ell-1$.
\end{proof}

Choose an edge $uv\in E(G)$ such that $|M_{uv}|\leq \ell-1$. By Claim \ref{claim3}, such $uv$ exists.
Set $G':=G\backslash \{u,v\}$ and $n':=|G'|$. Then $n'=n-2=4\ell -3$ is odd by Claim \ref{claim1}.

\begin{claim}\label{claim4}
The graph $G'$ contains $\ell-1$ pairwise disjoint induced subgraphs, each isomorphic to $P_3$.
\end{claim}
\begin{proof}
Set $\ell':=\ell-1$. Then $n'=4\ell' +1$ is odd, so $\ell'< \lceil\frac{n'}{4}\rceil$. 
Since $\kappa(G)\geq \lceil\frac{n}{2}\rceil$, we have $\kappa(G')\geq \lceil\frac{n}{2}\rceil-2= \lceil\frac{n'}{2}\rceil-1 \geq \lceil\frac{n'-1}{4}\rceil $.
Moreover, since $\omega(G)< \frac{n+3}{4}$, we have $\omega(G')< \frac{n+3}{4} \leq \lceil\frac{n'}{2}\rceil$.
Hence, by Lemma \ref{lem1}, $G'$ has $\ell-1$ pairwise disjoint induced subgraphs, each isomorphic to $P_3$.
\end{proof}

By Claim \ref{claim4}, there are 
$\ell-1$ pairwise disjoint induced subgraphs in $G'$, each isomorphic to $P_3$, denote the collection of such subgraphs  by $\mathcal{P}$.
Set $B:=V(G)-(V(\mathcal{P})\cup \{u,v\})$.
Then $|B|=\ell$ by Claim \ref{claim1}. Let $H$ be the graph obtained from $G$ by contracting the edge $uv$ and each induced path of $\mathcal{P}$ to a vertex. Since $\alpha(G)=2$, each vertex in $H$ obtained by contracting each member of $\mathcal{P}$ is complete to all other vertices in $H$. Hence, to prove $H$ contains a subgraph isomorphic to  $K^{\ell}_{\ell,\ell}=K^{\ell}_{\ell,\lceil\frac{n}{2}\rceil-\ell}$, it suffices to show that  $B\subseteq N_G[\{u,v\}]$ as $|B|=\ell$.
Without loss of generality we may assume that $\mathcal{P}$ is chosen with $|N_G[\{u,v\}]-V(\mathcal{P})|$ as large as possible.


\begin{claim}\label{claim5}
$B\subseteq N_G[\{u,v\}]$. 
\end{claim}
\begin{proof}
Suppose not. Let $b\in B-N_G[\{u,v\}]$. Since $|M_{uv}|\leq \ell-1=|\mathcal{P}|$ by the choice of $uv$, there must exist an induced 3-vertex-path $P \in \mathcal{P}$ such that $V(P) \subseteq N_G[\{u,v\}]$. Set $P:=a_1$-$a_2$-$a_3$. 
Since $\alpha(G)=2$, by symmetry we may assume that $a_1b\in E(G)$. Assume that $a_3b\in E(G)$. Replacing $P$ with $a_3$-$b$-$a_1$ get $\ell-1$ pairwise disjoint induced 3-vertex-paths $\mathcal{P}'$ in $G'$. Since $|N_G[\{u,v\}]-V(\mathcal{P}')|>|N_G[\{u,v\}]-V(\mathcal{P})|$, we get a contradiction to the choice of $\mathcal{P}$. Hence, to prove the claim, it suffices to show that $a_3b\in E(G)$.


Assume not. When $a_2b\notin E(G)$,  we can replace $P$ with $a_2$-$a_1$-$b$ to get a contradiction to the choice of $\mathcal{P}$. So $a_2b\in E(G)$. Under this case, we still can replace the $P$ with $a_3$-$a_2$-$b$ to get a contradiction to the choice of $\mathcal{P}$. Hence, $b$ is complete to $\{a_1,a_3\}$.
\end{proof}
\end{proof}




A graph $G$ is \emph{$k$-vertex-critical} if $\chi(G) = k$
and $\chi(G-v) < k$ for each $v\in V(G)$. If the complement of $G$ is connected, we say that $G$ is {\em anti-connected}.

\begin{theorem}{\rm(\cite{MS})}\label{thmMS}
Let $G$ be a $k$-vertex-critical graph such that $G$ is anti-connected. Then for any $x\in V(G)$, the graph $G-x$ has a $(k-1)$-coloring in which every color class contains at least {\rm2} vertices.
\end{theorem}

Following some ideas in \cite{MM}, we can prove the following result, implying that Theorem \ref{thm4} holds.
\begin{theorem}\label{thm-last}
Theorems \ref{thm3.5} and \ref{thm4} are equivalent.
\end{theorem}
\begin{proof}
Since $|V(G)|\leq\alpha(G)\chi(G)$ for any graph $G$, Theorem \ref{thm4} implies Theorem \ref{thm3.5}. Hence, it suffices to show that when Theorem \ref{thm3.5} holds, so does Theorem \ref{thm4}. Assume not.
Let $G$ be a counterexample to Theorem \ref{thm4} with $|V(G)|$ as small as possible. Then there exists an integer $\ell$ with $2\ell \leq \chi(G)$ such that $G$ contains no $K^{\ell}_{\ell,\chi(G)- \ell}$-minor. Since $\Delta(G)\geq \chi(G)-1$ implies that $G$ contains a $K^{1}_{1,\chi(G)-1}$-minor, we have $\ell\geq2$. 

We claim that $G$ is $\chi(G)$-vertex-critical. Assume not. Then $\chi(G-x)=\chi(G)$ for some vertex $x\in V(G)$. By the minimality choice of $G$, we have $$G-x \succeq_{m} K_{\ell',\chi(G-x)- \ell'}^{\ell'}=K_{\ell',\chi(G)- \ell'}^{\ell'}$$ for any positive integer $\ell'$ with $2\ell'\leq\chi(G-x)=\chi(G)$, which is a contradiction. Hence, $G$ is $\chi(G)$-vertex-critical.

When $2\chi(G)-1\leq|V(G)|\leq 2\chi(G)$, it follows from the fact $\chi(G)\geq\frac{|V(G)|}{2}$ that $\chi(G)=\lceil\frac{|V(G)|}{2}\rceil$, so $G$ is not  a counterexample to Theorem \ref{thm4} by Theorem \ref{thm3.5}, which is a contradiction. 
Hence, $|V(G)|\leq 2\chi(G)-2$ as $|V(G)|\leq 2\chi(G)$. Moreover, since $G$ is $\chi(G)$-vertex-critical, $G$ is not anti-connected by Theorem \ref{thmMS}.
Then there is a partition $(V_1,V_2)$ of $V(G)$ such that $V_1$ is complete to $V_2$. Set $G_1:=G[V_1]$ and $G_2:=G[V_2]$. Then $\chi(G)=\chi(G_1)+\chi(G_2)$.
By the minimality choice of $G$, we have $G_1 \succeq_{m} K_{\ell_1,\chi(G_1)- \ell_1}^{\ell_1}$ and
$G_2 \succeq_{m} K_{\ell_2,\chi(G_2)- \ell_2}^{\ell_2}$
for any positive integers $\ell_1$, $\ell_2$ with $2\ell_1\leq \chi(G_1)$ and $2\ell_2\leq \chi(G_2)$. Hence, when $G_1$ or $G_2$ is a clique, obviously we have $G \succeq_{m}K_{\ell,\chi(G)- \ell}^{\ell},$ a contradiction. So  neither $G_1$ nor $G_2$ is a clique. Moreover, when there are positive integers $\ell_1$, $\ell_2$ with $2\ell_1\leq \chi(G_1)$, with $2\ell_2\leq \chi(G_2)$, and with $\ell=\ell_1+\ell_2$, we have $$G \succeq_{m} K^{{\ell_1+\ell_2}}_{\ell_1+\ell_2,\chi(G_1)+\chi(G_2)- (\ell_1+\ell_2)}=K_{\ell,\chi(G)- \ell}^{\ell},$$ which is a contradiction. 
Hence, such $\ell_1$, $\ell_2$ do not exist, implying that $\chi(G)=2\ell$, and both $\chi(G_1)$ and $\chi(G_2)$ are odd.

Set $\ell_1:=\frac{\chi(G_1)-1}{2}$ and $\ell_2:=\frac{\chi(G_2)-1}{2}$. Then $\ell=\ell_1+\ell_2+1$. Since $V_1$ is complete to $V_2$, and $G$ is $\chi(G)$-vertex-critical, $G_i$ is $\chi(G_i)$-vertex-critical for each integer $1\leq i\leq2$. Moreover, since $G_i$ is not a clique, 
there are two nonadjacent vertices $u_i,v_i\in V(G_i)$ 
such that $$\chi(G_i\backslash\{u_i,v_i\})=\chi(G_i-u_i)=\chi(G_i-v_i)=\chi(G_i)-1=2\ell_i.$$ By the minimality choice of $G$, we have $G_i\backslash\{u_i,v_i\} \succeq_{m} K_{\ell_i, \ell_i}^{\ell_i}$. 
Then $G\backslash\{u_1,v_1,u_2,v_2\} \succeq_{m} K_{\ell-1,\ell-1}^{\ell-1}$ as $V_1$ is complete to $V_2$. Moreover, since $N_G[\{u_1,u_2\}]=N_G[\{v_1,v_2\}]=V(G)$, adding the two new vertices obtained by contracting edges $u_1u_2$ and $v_1v_2$ to a $K_{\ell-1,\ell-1}^{\ell-1}$-minor of $G\backslash\{u_1,v_1,u_2,v_2\}$, we obtain a $K_{\ell,\chi(G)- \ell}^{\ell}$-minor of $G$ as $\chi(G)=2\ell$, which is a contradiction.
\end{proof}

For graphs with independence number two, following a similar and simpler way as the proof of Theorem \ref{thm-last}, we can also prove that Conjecture {\rm\ref{conj1w}} is equivalent to Conjecture {\rm\ref{conj4}} for graphs with independence number two.

\section{Acknowledgments}
This research was partially supported by grants from the National Natural Sciences
Foundation of China (No. 11971111).
We would like to express our gratitude to two anonymous reviewers for their diligent review and valuable suggestions, which significantly enhanced the clarity and presentation of this paper.

	
	
	\vspace{5pt}


\begin{thebibliography}{99}
\bibitem{J}
J. Blasiak, A special case of Hadwiger's conjecture, J. Combin. Theory Ser. B, 97 (2007) 1056-1073.

\bibitem{FA}
F. Botler, A. Jim\'{e}nez, C.  Lintzmayer, A. Pastine, D.  Quiroz, M. Sambinelli, Biclique immersions in graphs with independence number 2, Eur. J. Combin., 122 (2024) 104042.

\bibitem{M}
M. Chudnovsky, Hadwiger's conjecture and seagull packing, Notices Amer. Math. Soc., 57 (6) (2010) 733-736.

\bibitem{MP}
M. Chudnovsky, P. Seymour, Packing seagulls, Combinatorica, 32 (2012) 251-282.

\bibitem{PH}
P. Duchet, H. Meyniel, On Hadwiger's number and the stability number,
Ann. Discrete Math., 13 (1982) 71-74.

\bibitem{HU}
H. Hadwiger, \"{U}ber eine Klassifikation der Streckenkomplexe, Vierteljschr. Naturforsch. Ges. Z\"{u}rich, 88 (1943) 133-143.

\bibitem{AN}
A.  Kostochka, N. Prince, Dense graphs have $K_{3,t}$-minors. Discrete Math., 310 (20) (2010) 2637-2654.

\bibitem{MM}
M.  Plummer, M. Stiebitz, B. Toft, On a special case of Hadwiger's conjecture, Discuss. Math. Graph Theory, 23 (2003) 333-363.

\bibitem{P}
P. Seymour, Hadwiger's conjecture, in: J.F. Nash Jr., M.Th. Rassias (Eds.), Open Problems in Mathematics, Springer International Publishing, 2016.

\bibitem{MS}
M. Stehl\'{e}k, Critical graphs with connected complements, J. Combin. Theory Ser. B, 89 (2) (2003) 189-194.


\bibitem{W}
D.  Woodall, List colourings of graphs, Surveys in combinatorics, (2001)  269-301.













\end{thebibliography}
\end{document}